\documentclass[11pt]{elsarticle}
\usepackage{amsfonts, amssymb, latexsym, epsfig} 

\usepackage{amscd,amssymb, mathtools, etoolbox}
\usepackage{verbatim}
\usepackage[mathscr]{euscript}
\usepackage{amsmath}
\input xy
\xyoption{all}

\makeatletter
\def\moverlay{\mathpalette\mov@rlay}
\def\mov@rlay#1#2{\leavevmode\vtop{%
   \baselineskip\z@skip \lineskiplimit-\maxdimen
   \ialign{\hfil$\m@th#1##$\hfil\cr#2\crcr}}}
\newcommand{\charfusion}[3][\mathord]{
    #1{\ifx#1\mathop\vphantom{#2}\fi
        \mathpalette\mov@rlay{#2\cr#3}
      }
    \ifx#1\mathop\expandafter\displaylimits\fi}
\makeatother

\usepackage{tikz, standalone}
\usetikzlibrary{shapes,arrows}

\usetikzlibrary{matrix,arrows,backgrounds,shapes.misc,shapes.geometric,patterns,calc,positioning}
\usetikzlibrary{intersections,patterns,arrows,chains,matrix,positioning,scopes,decorations.pathreplacing,decorations.pathmorphing,calc}

\usepackage[colorlinks=true, pdfstartview=FitV, linkcolor=blue, citecolor=blue, urlcolor=blue]{hyperref}
\usepackage{tikz, standalone, float}
\usepackage{fullpage,subcaption}
\usepackage{graphicx}
\usepackage{enumerate}
\def\VR{\kern-\arraycolsep\strut\vrule &\kern-\arraycolsep}
\def\vr{\kern-\arraycolsep & \kern-\arraycolsep}

\usepackage{tikz-cd}
\usepackage{extarrows}

\pgfdeclarelayer{edgelayer}
\pgfdeclarelayer{nodelayer}
\pgfdeclarelayer{background}
\pgfdeclarelayer{foreground}
\pgfsetlayers{background,foreground,edgelayer,nodelayer,main}

\usepackage{blkarray}

\newcommand{\be}{\begin{enumerate}}

\newproof{pf}{Proof}
\newtheorem{theorem}{Theorem}[section]
\newtheorem{prop}{Proposition}[section]
\newtheorem{lemma}[prop]{Lemma}

\newtheorem{definition}[prop]{Definition}

\newtheorem{rmk}{Remark}

\newtheorem{obs}{Observation}

\newtheorem{ex}{Example}

\DeclareMathOperator{\PEnd}{\underline{End}}

\newcommand{\Hom}{\operatorname{Hom}}

\newcommand{\Ext}{\operatorname{Ext}}

\newcommand{\ZZ}{\mathbb Z}

\newcommand{\C}{\mathcal{C}}

\makeatletter
\tikzset{join/.code=\tikzset{after node path={%
\ifx\tikzchainprevious\pgfutil@empty\else(\tikzchainprevious)%
edge[every join]#1(\tikzchaincurrent)\fi}}}
\makeatother
\tikzset{>=stealth',every on chain/.append style={join},
         every join/.style={->}}
\tikzstyle{labeled}=[execute at begin node=$\scriptstyle,
   execute at end node=$]
   
   \makeatletter
\tikzset{join/.code=\tikzset{after node path={%
\ifx\tikzchainprevious\pgfutil@empty\else(\tikzchainprevious)%
edge[every join]#1(\tikzchaincurrent)\fi}}}
\makeatother

\begin{document}
\title{Universal deformation rings of modules for  generalized Brauer tree algebras of polynomial growth}

\author[1]{David C. Meyer}
\author[2]{Roberto C. Soto}
\author[3]{Daniel J. Wackwitz-contact author}

\address[1]{Reed College, Mathematics \& Statistics Department, Portland, OR 97202: davidmeyer@reed.edu}
\address[2]{California State University, Fullerton, Mathematics Department, 800 N State College Blvd, Fullerton, CA 92831: rcsoto@fullerton.edu}
\address[3]{University of Wisconsin - Platteville, Mathematics Department, 1 University Plaza, Platteville, WI 53818: wackwitzd@uwplatt.edu}


\begin{abstract}

Let $k$ be an arbitrary field, $\Lambda$ be a $k$-algebra and $V$ be a $\Lambda$-module.  When it exists, the universal deformation ring $R(\Lambda,V)$ of $V$ is a $k$-algebra whose local homomorphisms to $R$ parametrize the lifts of $V$ up to $R\otimes_k \Lambda$, where $R$ is any complete, local commutative Noetherian $k$-algebra with residue field $k$.  Symmetric special biserial algebras, which coincide with Brauer graph algebras, can be viewed as generalizing the blocks of finite type $p$-modular group algebras.  Bleher and Wackwitz classified the universal deformation rings for all modules for symmetric special biserial algebras with finite representation type.  In this paper, we begin to address the tame case.  Specifically, let $\Lambda$ be any 1-domestic, symmetric special biserial algebra.  By viewing $\Lambda$ as generalized Brauer tree algebras and making use of a derived equivalence, we classify the universal deformation rings for those $\Lambda$-modules $V$ with stable endomorphism ring isomorphic to $k$.  The latter is a natural condition, since it guarantees the existence of the universal deformation ring $R(\Lambda,V)$.
\end{abstract}


\maketitle
\setcounter{tocdepth}{1}

\section{Introduction}


Let $k$ be an arbitrary field, $\Lambda$ a finite-dimensional $k$-algebra and $V$ a finitely generated $\Lambda$-module. It is natural to wonder if it is possible to lift $V$ to a module for $R\otimes_k \Lambda$, where $R$ is a complete local commutative Noetherian $k$-algebra with residue field $k.$ From Bleher and V{\'e}lez-Marulanda \cite[Proposition 2.1]{bleher2012velez} we know that if $\Lambda$ is self-injective and the stable endomorphism ring of $V$ is isomorphic to $k,$ then there exists a complete local commutative Noetherian $k$-algebra $R(\Lambda, V)$ which is universal in the sense that each lift of $V$ up to such an algebra comes from a unique local homomorphism from $R(\Lambda,V)$. That is, if $R$ is a complete local commutative Noetherian $k$-algebra with residue field $k$, the local homomorphisms from $R(\Lambda,V)$ to $R$ parametrize the isomorphism classes of all lifts $V$ up to $R\otimes_k \Lambda$.  The ring $R(\Lambda, V)$ is the universal deformation ring of $V$, and in the above sense it has much to say about to which algebras $V$ can be lifted.

In the case that $G$ is a finite group, $k$ has positive characteristic and $\Lambda$ is the group ring $kG$, the question of lifting modules has a long tradition. Indeed, not only are lifts of $V$ to $k$-algebras studied, but also lifts to arbitrary complete local commutative Noetherian rings with residue field $k$. This work has contributed to a better understanding of the connections between representations of groups in characteristic $p$ and characteristic 0. For example, Green \cite{green1959lifting}  showed that if $k$ is the residue field of a ring of $p$-adic integers $\mathcal{O}$ and if there are no non-trivial 2-extensions of a finitely generated $kG$-module $V,$ then $V$ can be lifted to $\mathcal{O}.$ This led Auslander, Ding, and Solberg \cite{auslander1993lifting} to consider more general lifting problems of algebras over Noetherian rings. Then, as a consequence of his work on the lifts of tilting complexes, Rickard \cite{rickard1991lifting} generalized Green's result to modules for arbitrary finite rank algebras over complete local commutative Noetherian rings. 

More recently, much has been done to classify the universal deformation rings of modules with stable endomorphism ring isomorphic to $k$, for many particular special biserial algebras.  Special biserial algebras were first studied by Gelfand and Ponomarev \cite{gel1968indecomposable}, and appear in the characteristic $p$ representation theory of finite groups \cite{alperin1993local}, as the representation-finite blocks of the group algebra $kG$ are Morita equivalent to special biserial algebras arising from Brauer trees with one distinguished vertex \cite{dade1966blocks, janusz1969indecomposable, kupisch1968projektive, alperin1993local}.  While special biserial algebras are usually of infinite representation type, their modules can be studied systematically since all nonprojective indecomposable modules correspond to either strings or bands.  In \cite{velezspecialbiserial, velezspecialbiserial2}, Vélez-Marulanda and his collaborators study the universal deformation rings for string modules for two families of special biserial algebras.  Bleher and Talbott \cite{shannondihedral} work with modules for a related class of algebras, the algebras of dihedral type with polynomial growth. While these algebras are not special biserial, they are closely related in that they become special biserial after quotienting out their socles.  In \cite{bleher2019wackwitz}, Bleher and Wackwitz classify the universal deformation rings for modules for both Brauer tree and Nakayama algebras, including those with arbitrary stable endomorphism rings.  We will build on these results, focusing on a different extension of Brauer tree algebras.  In this paper, we consider the case when $\Lambda$ is a domestic generalized Brauer tree algebra over an algebraically closed field $k$ of arbitrary characteristic.  Brauer graph algebras are given combinatorially by a graph and a labeling of its vertices (see \cite{schroll2018notes} for a more comprehensive introduction to Brauer graph algebras).  We study a subclass of these algebras, the generalized Brauer tree algebras of polynomial growth. These algebras are defined by properties of their corresponding graphs, but they coincide exactly with the 1-domestic symmetric, special biserial algebras \cite{roggenkamp24biserial, schroll2015trivial}.   Our main results in this paper apply to any one of these algebras and to any module with a stable endomorphism ring isomorphic to $k$, including band modules.

We summarize our results below; for more precise statements see Propositions \ref{exceptional}, \ref{nonperiodic}, and \ref{homogeneous}. Note that details on the components of the stable Auslander-Reiten quiver of Brauer tree algebra, $\Lambda,$ referenced below can be found in Section \ref{brauer}.

\begin{theorem} \label{mainthm} Let $k$ be an algebraically closed field of arbitrary characteristic, and let $\Lambda$ be a generalized Brauer tree algebra with $e$ edges which has two vertices of multiplicity 2, and all other vertices have multiplicity one. Let $\mathfrak{C}$ be a connected component of the stable Auslander-Reiten quiver of $\Lambda$, and let $V$ be a module belonging to $\mathfrak{C}$ with 
$\underline{\mathrm{End}}_\Lambda(V)\cong k$.
\begin{enumerate}[(i)]
  \item If $\mathfrak{C}$ is an exceptional tube, then $R(\Lambda, V)$ is isomorphic to $k$ or $k[[t]]$.
  \item If $\mathfrak{C}$ is a non-periodic component and $e \geq 2$, then $R(\Lambda,V)$ is isomorphic to $k$ or $k[[t]]/(t^2)$. If $e=1$, then $R(\Lambda,V) \cong k[[t_1,t_2]]/(t_1^2-t_2^2, t_1t_2)$.
  \item If $\mathfrak{C}$ is a homogeneous tube of rank one, then $R(\Lambda,V) \cong k[[t]]$.
\end{enumerate}
\label{Brauer star theorem}

\end{theorem}

Theorem \ref{mainthm} pertains to any 1-domestic symmetric special biserial algebra, since these algebras are precisely those that correspond to a graph with the stated properties when viewed as Brauer graph algebras (see \cite{schroll2018notes}).  This paper is organized as follows. In Section \ref{udr}, we introduce the universal deformation ring of a finitely generated $\Lambda$-module $V$.  In Section \ref{brauer}, we assume $k$ is algebraically closed, and we review generalized Brauer tree algebras briefly, highlighting the results that will be instrumental in our analysis. In Section \ref{StableEndRing}, we determine which modules have stable endomorphism ring isomorphic to $k$, and use these results to prove Theorem \ref{mainthm} in Section \ref{MainResults}.

\section{Universal deformation rings}
\label{udr}

Universal deformation rings were introduced by Mazur to systematically study the $p$-adic lifts of finite-dimensional irreducible mod $p$ representations of profinite Galois groups \cite{mazur}.  In this paper, we consider only lifts to $R\otimes_k \Lambda$, where $R$ is a complete local commutative Noetherian $k$-algebra, as formulated by Bleher and V{\'e}lez-Marulanda \cite{bleher2012velez}.  For a more extensive treatment, on universal deformation rings, see \cite{mazur}, \cite{bleher2012velez} or \cite{deSmit}.  The modules we consider are unital left $\Lambda$-modules, which are finite-dimensional over $k$, and $\Lambda$ itself is a finite-dimensional $k$-algebra.

Let  $k$ be a field of arbitrary characteristic and let $\hat{\mathcal{C}}$ be the category whose objects are tuples $(R, {\pi}_R)$, where $R$ is a complete local commutative Noetherian $k$-algebra with residue field $k$, and ${\pi}_R$ is a projection to the residue field.  If $\Lambda$ is a finite-dimensional $k$-algebra and $V$ is a finitely generated $\Lambda$-module, a \emph{lift} of $V$ over an object $R$ in $\hat{\mathcal{C}}$ is a pair $(M,\phi)$, where $M$ is a finitely generated $R\otimes_k\Lambda$-module that is free over $R$, and $\phi$ is an isomorphism of $R\otimes_k\Lambda$-modules 
$$k\otimes_R M\xrightarrow{\phi} V.$$
Thus, a lift of $V$ to $R$ is an $R\otimes_k\Lambda$-module $M$ that is isomorphic to $V$ when one extends scalars to $k\Lambda \cong \Lambda$.  We say two lifts $(M,\phi)$ and $(M'{\phi}')$ are isomorphic, if there is an $R\otimes_k\Lambda$-module isomorphism $f:M\to M'$ with $\phi={\phi}'\circ (id\otimes f)$.  A deformation of $V$ over $R$ is the isomorphism class of a lift of $V$ to $R$, and the set of all deformations (to $R$) is denoted by $\mathrm{Def}_\Lambda(V,R)$.

The \emph{deformation functor} $\hat{F}_V$ is a functor from $\hat{\mathcal{C}}$ to $ \mathrm{Sets}$, which categorically encodes information about all the lifts of $V$ up to $R$, for all $R \in \hat{\mathcal{C}}$.  At the level of objects, the deformation functor sends $R \in \hat{\mathcal{C}}$ to  $\mathrm{Def}_\Lambda(V,R) \in \textrm{Sets}.$  That is, to say, each ring $R$ is associated by the deformation functor to the the set of isomorphism classes of lifts of $V$ up to $R$.  If there exists an object $R(\Lambda,V)$ in $\hat{\mathcal{C}}$, such that $R(\Lambda,V)$ \emph{represents} the functor $\hat{F}_V$ in the sense that $\hat{F}_V$ is naturally isomorphic to $\mathrm{Hom}_{\hat{\mathcal{C}}}(R(\Lambda,V),-)$, we call $R(\Lambda,V)$ the \emph{universal deformation ring} of $V$.  Thus, if $V$ has a universal deformation ring $R(\Lambda,V)$, then for every $R \in \hat{\mathcal{C}}$, every isomorphism class of a lift of $V$ to $R$ comes from a unique local ring homorphism from $R(\Lambda,V)$ to $R$.  In this sense, universal deformation rings contain information about all lifts of $V$ up to local rings in $\hat{\mathcal{C}}$. It was shown in (\cite[Proposition 2.1]{bleher2012velez}) that every finitely generated $\Lambda$-module $V$ with stable endomorphism ring, $\PEnd_{\Lambda}(V),$ isomorphic to $k$ has a universal deformation ring, and that Morita equivalence preserve universal deformation rings (\cite[Proposition 2.5]{bleher2012velez}). Let $\Omega$ denote the syzygy functor and suppose that $\Lambda$ is self-injective. Then  we have the following result:


\begin{theorem} \label{isomk}
(\cite[Theorem 2.6]{bleher2012velez}) Let $\Lambda$ be a finite dimensional self-injective $k$-algebra, and suppose $V$ is a finitely generated $\Lambda$-module whose stable endomorphism ring, $\PEnd_{\Lambda}(V)$, is isomorphic to $k.$
 \begin{itemize}
   \item[(i)] The module $V$ has a universal deformation ring $R(\Lambda, V).$
   \item[(ii)] If $P$ is a finitely generated projective $\Lambda$-module, then $\PEnd_{\Lambda}(V \oplus P)) \cong k$ and $R(\Lambda, V) \cong R(\Lambda, V \oplus P).$
   \item[(iii)] If $\Lambda$ is moreover a Frobenius algebra, then $\PEnd_{\Lambda}(\Omega(V)) \cong k$ and $R(\Lambda, V)  \cong R(\Lambda, \Omega(V)).$
 \end{itemize}
\end{theorem}

\begin{rmk}
When $\Lambda$ and $V$ satisfy the hypothesis of Theorem \ref{isomk}, then in particular, $V$ has a universal deformation ring $R(\Lambda, V).$ By \cite[Proposition 2.1]{bleher2012velez}, the tangent space $t_V$ of the deformation function $\hat{F}_V$ which is defined as $t_V = \hat{F}_V(k[\epsilon]/(\epsilon^2)),$ is isomorphic to $\Ext^1_{\Lambda}(V, V)$ as a $k$-vector space. Since $R(\Lambda, V)$ represents $\hat{F}_V$, then $$\Ext^1_{\Lambda}(V, V) \cong t_V \cong \Hom_{\hat{C}}(R(\Lambda, V), k[\epsilon]/(\epsilon^2)) \cong \Hom_k(\mathfrak{m}_{\Lambda, V}/\mathfrak{m}^2_{\Lambda, V},k),$$ where $\mathfrak{m}_{\Lambda, V}$ denotes the maximal ideal of $R(\Lambda, V).$ Thus, if $\Ext^1_{\Lambda}(V, V)$ has $k$-dimension $r,$ then $R(\Lambda, V)$ is isomorphic to a quotient algebra of the power series algebra $k[[t_1, \ldots, t_r]]$ in $r$ commuting variables where $r$ is minimal with this property.
\end{rmk}

\section{Domestic Generalized Brauer Tree Algebras}
\label{brauer}

In this section we consider domestic generalized Brauer tree algebras. For a more detailed description of Brauer graph algebras see \cite{schroll2018notes}.  In all that follows, let $k$ be an algebraically closed field of arbitrary characteristic. We recall that an algebra $\Lambda$ is \emph{domestic} if there is a finite number of $\Lambda$-$k[x]$-bimodules $M_i$ which are finitely generated free right $k[x]$-modules such that all but a finite number of isomorphism classes of indecomposable $\Lambda$-modules of dimension $n$, for every dimension $n \geq 1,$ are of the form $M_i \otimes_{k[x]} V$ for some $i$ and some indecomposable $k[x]$-module $V$ \cite{ringeltame}. If the minimal numbers of such bimodules is $d$, then $\Lambda$ is $d$-domestic. In \cite{erdmann1992auslander}, Erdmann and Skowro{\'n}ski described the stable Auslander-Reiten quiver of a special biserial algebra of domestic representation type.  Then, in \cite{duffield2018auslander}  Duffield was able to give a more detailed  description of the components of the Auslander-Reiten quiver for these algebras. For convenience we restate parts of these results here.

\begin{theorem}
\cite[Theorem 2.1]{erdmann1992auslander}  Let $\Lambda$ be a self-injective special biserial algebra. The following are equivalent:
 \begin{enumerate}[(i)]
    \item $\Lambda$ is representation-infinite domestic.
    \item $\Lambda$ is representation-infinite of polynomial growth.
    \item $\prescript{}{s} \Gamma_{\Lambda}$ has a component of the form $\ZZ \tilde{A}_{p, q}.$
    \item $\prescript{}{s} \Gamma_{\Lambda}$ is infinite but has no component of the form $\ZZ A_{\infty}^{\infty}$
    \item All but a finite number of components of $\Gamma_{\Lambda}$ are of the form $\ZZ A_{\infty}/ \langle \tau \rangle$
    \item There are positive integers $m, p,$ and $q$ such that $\prescript{}{s} \Gamma_{\Lambda}$ is a disjoint union of $m$ components of the form $\ZZ \tilde{A}_{p, q}$, $m$ components of the form  $\ZZ A_{\infty}/ \langle \tau^p \rangle$, $m$ components of the form $\ZZ A_{\infty}/ \langle \tau^q \rangle,$ and infinitely many components of the form $\ZZ A_{\infty}/ \langle \tau \rangle.$
 \end{enumerate}
\end{theorem}

\begin{theorem}
\cite[Theorem 4.4]{duffield2018auslander}
Let $\Lambda$ be a Brauer graph algebra constructed from a graph $G$ of $n$ edges and suppose $\prescript{}{s} \Gamma_{\Lambda}$ has a $\ZZ \tilde{A}_{p, q}$ component.
\begin{enumerate}[(i)]
  \item If $\Lambda$ is 1-domestic, then $p+q=2n.$
  \item If $\Lambda$ is 2-domestic, then $p+q=n.$
\end{enumerate}
Furthermore, if $G$ is a tree, then $p=q=n.$
\end{theorem}

Let $\Lambda$ be a representation-infinite Brauer graph algebra with Brauer graph $G$.  Then, by \cite{roggenkamp24biserial} we know that $\Lambda$ is 1-domestic if and only if $\Lambda$ satisfies one of the following conditions:\\
(D1) $G$ is a tree with exactly two vertices having multiplicity 2 and all other vertices having multiplicity 1;\\
(D2) $G$ is a graph with a unique cycle of odd length and all vertices having multiplicity 1.\\
\indent We say that a Brauer graph algebra $\Lambda$ is a {\emph{generalized Brauer tree algebra}} if the underlying graph of $\Lambda$ is a tree with at least two vertices with multiplicity greater than one.  In particular, Brauer graph algebras satisfying the property (D1) above are examples of generalized Brauer tree algebras.  We now consider a particular generalized Brauer tree algebra described as follows.  Fix $e \in \ZZ^+$ and let $\Lambda$ be a generalized Brauer tree algebra with graph $G$ with $e$ edges all incident with a common vertex, such that the two vertices incident with edge $e$ have multiplicity two and all other vertices have multiplicity 1 (see Figure \ref{fig:GenBrauerStar}). The corresponding Brauer graph algebra has quiver $Q_e, e \in \ZZ^+$ (see Figure \ref{fig:OurQuiver}).  Now, set $\Lambda = kQ_e/I$ where the ideal $I$ is described in Figure \ref{fig:OurQuiver}.

Since universal deformation rings are preserved by derived equivalences \cite{rickard1991derived, bleher2017velez}, by Kauer \cite{kauer1998derived} we may reduce to the case when $\Lambda = kQ_e/I$, since all Brauer graph algebras of type (1) as described above, are derived equivalent to $\Lambda$ for some choice of $e$.  That is, any 1-domestic symmetric special biserial algebra is derived equivalent to an algebra given by the generalized Brauer star in Figure \ref{fig:GenBrauerStar} and so for our analysis, we may reduce to $\Lambda$ as given above.

\begin{figure}
    \centering
    \includegraphics[scale = 0.5]{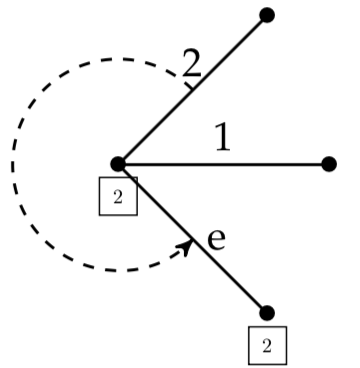}
    \caption{Generalized Brauer star $G$ having $e$ edges, $e \in \ZZ^+,$ and the two vertices incident with edge $e$ having multiplicity 2}
    \label{fig:GenBrauerStar}
\end{figure}

\begin{figure}
    \centering
    \includegraphics[scale = 0.4]{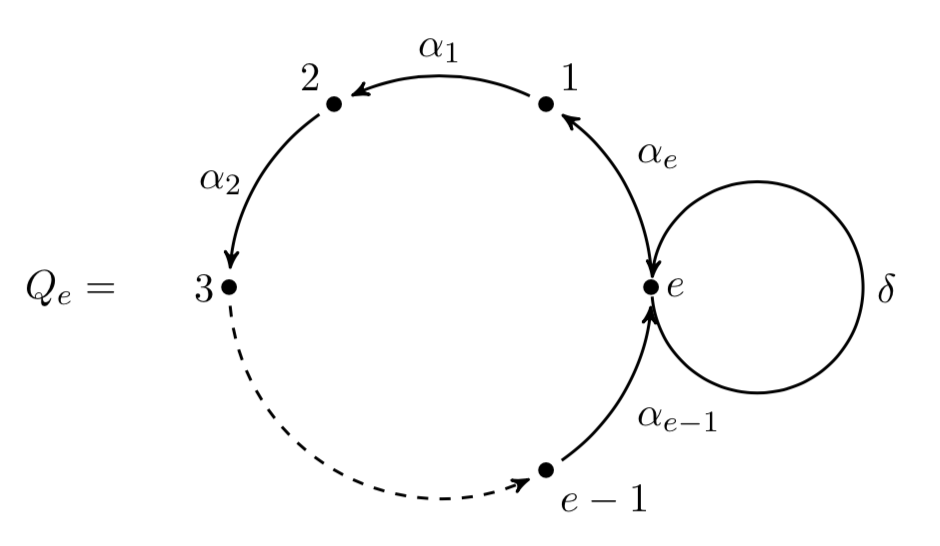}
    ~
\includegraphics[scale=0.5]{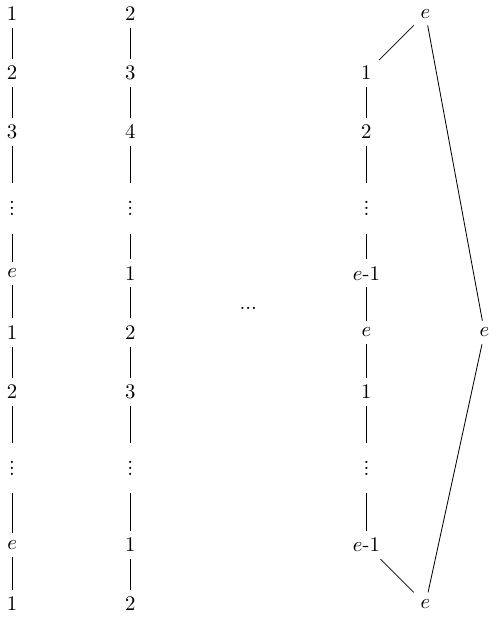}
    \caption{On the left is the quiver $Q_e$ associated with the generalized Brauer star having $e$ edges, with the two vertices incident with edge $e$ having multiplicity 2.  In $kQ_e/I$, the relations generating $I$ are $\alpha_i(\alpha_{i-1}  \cdots \alpha_1 \alpha_e \cdots \alpha_{i})^2$ for all $i$ with $1\leq i\leq e$, $(\alpha_{e-1} \cdots \alpha_1 \alpha_e)^2-\delta^2$, $\alpha_e \delta^2$, $\delta^2 \alpha_{e-1}$ and $\delta^3$.  On the right is a picture of the radical layers of the indecomposable projectives for $kQ_e/I$.  Note that all are uniserial except for $P_e$.}
    \label{fig:OurQuiver}
\end{figure}

Furthermore, we have that the stable Auslander-Reiten quiver of $\Lambda,$ denoted by $\prescript{}{s} \Gamma_{\Lambda},$ is the disjoint union of 1 component of the form $\ZZ \tilde{A}_{e, e}$ (non-periodic component),  2 components of the form $\ZZ A_{\infty}/ \langle \tau^e \rangle$ (exceptional tubes),  and infinitely many components of the form $\ZZ A_{\infty}/ \langle \tau \rangle$ (homogeneous tube of rank one).

In addition, for all $e \in \ZZ^+$ and for each algebra $\Lambda = kQ_e/I,$ we have $\Lambda/\textrm{soc}(\Lambda)$ is a string algebra, thus all indecomposable non-projective $\Lambda$-modules are obtained combinatorially by string and band modules (see \cite{butler1987auslander}). Moreover, the $\Lambda$-module homomorphisms between string and band modules have been explicitly described in \cite{krause1991maps}, though we may always view band modules explicitly as representations of the bound quiver $Q_e$ satisfying the relations in $I$ (see Figure \ref{fig:BandAsRep}).

\subsection{Band modules for Brauer Tree Algebras}
\label{bandmodules}
A few of the results in Section \ref{StableEndRing} depend on the unique structure of band modules for Brauer tree algebras. We highlight some important results in this subsection. For more on band modules see \cite{krause1991maps}.  By inspection, the only band modules correspond to the the word $\alpha_{e-1} \alpha_{e-2}...\alpha_1 \alpha_e {\delta}^{-1}$.  For example, when $e=3$ this corresponds to the band module $B(n,\lambda)$ given by
$$\begin{tikzpicture}[description/.style={fill=white,inner sep=2pt}]
\matrix (m) [matrix of math nodes, row sep=3em,
column sep=2.5em, text height=1.5ex, text depth=0.25ex]
{ k^n & k^n\\ 
k^n &  k^n \\ };
\path[->,font=\scriptsize]

(m-1-1) edge node[auto] {$J_n(\lambda)$} (m-1-2)
(m-1-1) edge node[left] {$I_n$} (m-2-1)
(m-2-1) edge node[auto] {$I_n$} (m-2-2)
(m-2-2) edge node[auto] {$I_n$} (m-1-2);
\end{tikzpicture}$$
where $I_n$ is the identity matrix and $J_n(\lambda)$ denotes an $n$ by $n$ Jordan block corresponding to $\lambda \in k^*$. See Figure \ref{fig:BandAsRep} for the band module as a representation of $\Lambda$. 

\begin{figure}
\centering
 \includegraphics[scale=.22]{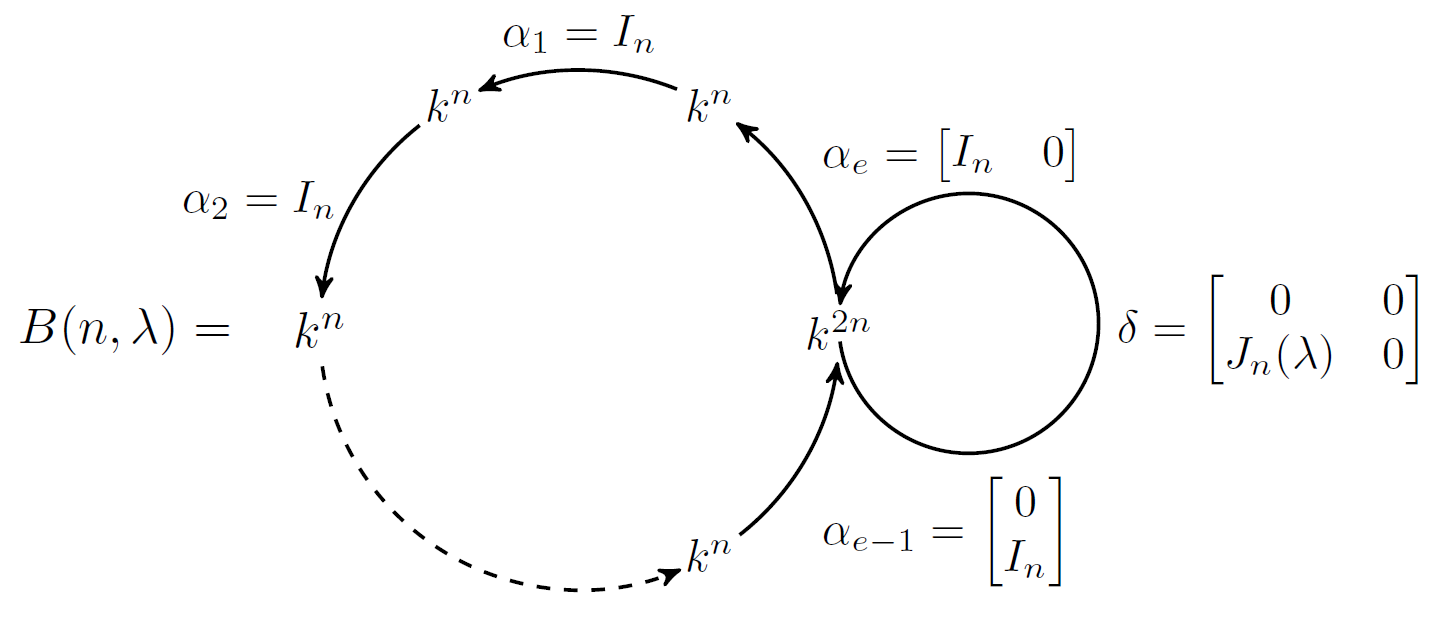}
 \caption{The band module $B(n,\lambda)$ as a representation of the bound quiver $Q_e$.  In the above, $0$ denotes the $n$ by $n$ zero matrix.}
 \label{fig:BandAsRep}
\end{figure}

We now compute the endomorphism ring of a band module.  Let $f = \{f(1), f(2), \ldots, f(e) \}$ be an arbitrary endomorphism, where $f(i)$ is a $k$-endomorphism of $B(n,\lambda)(i) \cong k^n$, and $$f(e) = \begin{pmatrix} A & B\\C & D\\ \end{pmatrix}, $$ where $A, B, C$ and $D$ are $n \times n$ matrices in $k$.

Commutativity along $\alpha_1, \alpha_2, \ldots, \alpha_n$ yields that $B=0$ and $$ A=f(1)=f(2)=\cdots = f(e-1) =D.$$  Commutativity along $\delta$ corresponds to the statement that $J_n(\lambda) \cdot A = D \cdot J_n (\lambda) = A \cdot J_n(\lambda)$.  Thus, letting $A'$ denote an $n$ by $n$ lower triangular matrix that is constant on the diagonals, and $C'$ denote an arbitrary $n$ by $n$ matrix, we have that $\text{dim}_k({\text{End}}_{\Lambda}(B(n,\lambda)))=n^2+n$ and 
$${\text{End}}_{\Lambda}(B(n,\lambda))= 
\begin{cases} k[x]/(x^2), & \text{ if } n=1\\
\begin{pmatrix} A' & 0\\ C' & A' \end{pmatrix}, & \text{ if } n >1.
\end{cases} $$
Note that by a similar argument $\text{dim}_k({\text{Hom}}_{\Lambda}(B(m,\lambda),B(n,\mu)))$ is $n^2$ if $\lambda \neq \mu$.

We will now classify all  indecomposable $\Lambda$-modules with stable endomorphism ring isomorphic to $k$  by analyzing each component of the stable Auslander-Reiten quiver of $\Lambda$ and thus obtain $\Lambda$-modules guaranteed to have universal deformation rings.

\section{Modules with stable endomorphism ring isomorphic to k}
\label{StableEndRing}

Recall that $\Lambda$ is a generalized Brauer star algebra with $e$ edges labeled $1$ through $e$, where the two vertices incident with edge $e$ have multiplicity two, and all other vertices have multiplicity one (see Figure \ref{fig:GenBrauerStar}).  In particular, $\Lambda=kQ_e/I$, where $Q_e$ and $I$ are as in Figure \ref{fig:OurQuiver}.

We begin by stating an obvious lemma: 

\begin{lemma}
\label{heightlemma}
Let $e \geq 2$ and let $M$ be a module for $kQ_e$ with $rad^{(e+1)} \cdot M=0$.  Let $f \in End_{\Lambda}(M)$ which factors through a projective.  Then $im (f) \subseteq soc(M)$.

\begin{pf}
Let $r:M \to P_i$ be any homomorphism.  Note that the height of $P_i$ is $2e+1$ for $i=1, 2, ... e.$ Hence the result follows.
\end{pf}
\end{lemma}

\begin{lemma} \label{stringmods}
Let $\mathfrak{C}$ be an exceptional tube of the stable Auslander-Reiten quiver of $\Lambda$.
\begin{enumerate}[(i)]

\item If $e \geq 2$, the modules in $\mathfrak{C}$ with stable endomorphism ring isomorphic to $k$ are the string modules $M(1_{e-1}),$ denoted as $S_{e-1}$, $M(\alpha_{e-2} \cdots \alpha_i)$ for $1 \leq i \leq e-2$, and $M(\alpha_{e-2} \cdots \alpha_1 \alpha_e \delta^{-1})$ along with their $\Omega$-translates.

\item If $e = 1$, the only modules $V$ in $\mathfrak{C}$ with stable endomorphism ring isomorphic to $k$ are the string modules $M(\alpha)$ and $M(\delta)=\Omega(M(\alpha))$.
\end{enumerate}

\end{lemma}
\begin{pf}
(i) Suppose $e \geq 2.$ We only need to consider the indecomposable $\Lambda$-modules starting with the simple module $S_{e-1}$ and the $e-1$ $\Lambda$-modules arrived at by adding right hooks to this module, as described in \cite{butler1987auslander}.  We denote these modules by $M(\alpha_{e-2} \cdots \alpha_i)$ for $1 \leq i \leq e-2$, and $M(\alpha_{e-2} \cdots \alpha_1 \alpha_e \delta^{-1})$ respectively.  Note that each of these modules will be in a different $\Omega$-orbit, which coincides with a different ``layer'' of the component, and these include all the $\Omega$-orbits in both periodic components. We begin by showing that each of these modules has stable endomorphism ring isomorphic to $k$, and then show that modules in all other ``layers'' have stable endomorphism rings whose $k$-dimensions are greater than 1.

First, for the modules $S_{e-1}$ and $M(\alpha_{e-2}\alpha_{e-3} \cdots \alpha_i)$ for $1\leq i\leq e-2,$ up to a scalar, the only endomorphism is the identity map, thus the stable endomorphism ring for each of these must be $k.$ Moreover, the module $M = M(\alpha_{e-2} \alpha_{e-3} \cdots \alpha_1 \alpha_e \delta^{-1})$ has a two-dimensional endomorphism ring generated by the identity map and the map sending the simple $S_e = M/\textrm{rad}(M)$ to the simple $S_e = \textrm{soc}(M).$  Since the latter map factors through the projective $P_e$ by first including the quotient module $M(\delta^{-1})$ into $\textrm{soc}(P_e),$ and then projecting $P_e$ onto $M$, ${\PEnd}_{\Lambda}(M)\cong k.$

We now show that for all $n \geq 1,$ the $\Lambda$-modules $M(\alpha_{e-2} \cdots \alpha_1 \alpha_e \delta^{-1}(\alpha_{e-1} \alpha_{e-2} \cdots \alpha_1 \alpha_e \delta^{-1})^{n}),$ denoted by $M_{n,0},$ and  $M(\alpha_{e-2} \cdots \alpha_1 \alpha_e \delta^{-1}(\alpha_{e-1} \alpha_{e-2} \cdots \alpha_1 \alpha_e \delta^{-1})^{n-1}\alpha_{e-1}\cdots \alpha_i)$ for $1\leq i \leq e-1,$ denoted by $M_{n,i},$ do not have stable endomorphism isomorphic to $k.$  First, consider $M_{n,0}$ for fixed $n \in \ZZ^+$.  Then, the map sending the quotient of $M(\alpha_{e-2} \cdots \alpha_1 \alpha_e \gamma^{-1})$ to the submodule of the same form does not factor through the projective by Lemma \ref{heightlemma}.

Now, fix $n \in \ZZ^+$ and $i \in \{1, \ldots, e-1\}.$ Then, the identity map of $M_{n,i}$ does not factor through a projective $\Lambda$-module. Now for $i \in \{1, \ldots, e-2\}$ consider the submodule of $M_{n,i}$ of the form $M(\alpha_{e-2} \cdots \alpha_i).$ This module is also a quotient module of $M_{n,i},$ and the endomorphism from the quotient module to the submodule does not factor through a projective by Lemma \ref{heightlemma}. Similarly, when $i=e-1$, the map sending the simple $S_{e-1} \in \textrm{rad}(M)$ to the simple $S_{e-1}$ which is a summand of the socle will also not factor through a projective.

\vspace{12px}
\noindent (ii) For $e=1$ we notice that the only modules to consider are the following:
\begin{enumerate}[i.]
 \item $M(\alpha);$ and
 \item $M((\alpha \delta^{-1})^n\alpha)$ for $n \geq 1.$
\end{enumerate}

For $M(\alpha)$ we note that there are only two maps and only the non-identity endomorphism factors through the projective $\Lambda$-module thus ${\PEnd}_{\Lambda}(M(\alpha)) \cong k.$ Fix $n \geq 1$ and consider the module $M((\alpha \delta^{-1})^n\alpha)$. The identity map does not factor through the projective and we note that $M(\alpha)$ is both a submodule and quotient module of $M((\alpha \delta^{-1})^n\alpha).$ By Lemma \ref{heightlemma} the map sending this quotient module of $M((\alpha \delta^{-1})^n\alpha)$ to the submodule $M(\alpha)$ does not factor through a projective. Thus ${\PEnd}_{\Lambda}(M((\alpha \delta^{-1})^n\alpha))$ is not isomorphic to $k.$
\end{pf}

\begin{lemma} \label{nonperiodicmods}
Let $\mathfrak{C}$ be the non-periodic component of the stable Auslander-Reiten quiver of $\Lambda$. Then there are precisely $e$ $\Omega$-orbits of modules in $\mathfrak{C}$, and all modules in these $\Omega$-orbits have stable endomorphism ring isomorphic to $k$.
\end{lemma}

\begin{pf}
Constructing the component $\mathfrak{C}$ we note that every module in $\mathfrak{C}$ is an $\Omega$-translate of the simple module $S_e$ or one of the uniserial modules $M(\alpha_{e-1} \alpha_{e-2} \cdots \alpha_i)$ for $1 \leq i \leq e-1.$ But for any of these modules, up to a scalar, the only endomorphism is the identity map. Hence the stable endomorphism ring for each of these modules is isomorphic to $k$.
\end{pf}

\begin{lemma}\label{bandstablelemma}
Let $\mathfrak{C}$ be the tube of rank 1 of the stable Auslander-Reiten quiver of $\Lambda$. Then the module(s) $B(n,\lambda)$ of $\mathfrak{C}$ have stable endomorphism ring isomorphic to $k$ if and only if $n = 1$ and $\lambda^2\neq -1.$
\end{lemma}

\begin{pf}

First consider $n=1$.  We show that if an endomorphism of $B(1, \lambda)$ factors through a projective, then it factors through $\bigoplus P_e$.
Recall, that for $U$ indecomposable, $\text{dim}_k({\text{Hom}}_{\Lambda}(P_i,U))$ is the multiplicity of $S_i$ in a composition series for $U$.  By self-injectivity, the same holds for $\text{dim}_k({\text{Hom}}_{\Lambda}(U,P_i))$.  Thus, ${\text{Hom}}_{\Lambda}(B(1,\lambda),P_i) = {\text{Hom}}_{\Lambda}(P_i,B(1,\lambda)) = k$, for $1 \leq i \leq e-1$, and $k^2$ for $i=e$.  It is easy to see that ${\text{Hom}}_{\Lambda}(P_i,B(1,\lambda)) \circ {\text{Hom}}_{\Lambda}(B(1,\lambda),P_i)) = 0$ for $1 \leq i \leq e-1.$

Now we will show that inclusion of the radical quotient of $B(1,\lambda)$ into its socle factors through $P_e$ if and only if $\lambda^2 \neq -1$. To do so, define a basis $\left\{b_{i,j}\right\}$ of $P_e$ where $b_{e,1}k \cong P_e/\text{rad}P_e, b_{e,2} = \alpha_{e-1} \ldots \alpha_1 \alpha_e(b_{e,1}), b_{e,3}=\delta(b_{e,1}), b_{e,4}=\delta^2(b_{e,1}), b_{1,j} = \alpha_e(b_{e,j})$ for $j \in \{1,2\}$, and $b_{i,j} = \alpha_{i-1}(b_{i-1},j)$ for $2 \leq i \leq e-1$ and $j \in {1,2}$. Also, define a basis $\left\{c_i\right\}$ of $B(1,\lambda)$ where $c_ek \cong B(1,\lambda)/\text{rad}B(1,\lambda),\\ c_{e+1}=\frac{1}{\lambda}\delta({c_e}), c_1=\alpha_e(c_e)$ and $c_i = \alpha_{i-1}(c_{i-1})$ for $2 \leq i \leq e-1$. We now define linearly independent maps $r_1, r_2: B(1,\lambda) \to P_3$ and $s_1, s_2:P_3 \to B(1,\lambda)$ which define bases for ${\text{Hom}}_{\Lambda}(B(1,\lambda),P_i)$ and ${\text{Hom}}_{\Lambda}(P_i,B(1,\lambda))$, respectively.
\begin{itemize}
\item The map $r_1$ acts by $c_e \mapsto b_{e,2}+\lambda b_{e,3}$, $c_i \mapsto b_{i,2}$ for $1 \leq i \leq e-1$, and $c_e+1 \mapsto b_{e,4}.$
\item The map $r_2$ acts by $c_e \mapsto b_{e,4}$ and all other basis elements are sent to $0$.
\item The map $s_1$ acts by $b_{e,1} \mapsto c_e$, $b_{e,2} \mapsto c_{e+1}$, $b_{e,3} \mapsto \lambda c_{e+1}$, $b_{i,1} \mapsto c_i$ for $1 \leq i \leq e-1$ and all other basis elements are sent to $0$.
\item The map $s_2$ acts by $b_{e,1} \mapsto c_{e+1}$ and all other basis elements are sent to $0$.
\end{itemize}

Note that $s_1 \circ r_2, s_2 \circ r_2$ and $s_2 \circ r_1$ are identically zero, while $s_1 \circ r_1$ sends $c_e$ to $(\lambda^2+1) c_{e+1}$ and all other basis vectors to zero, which is the map sending the radical quotient of $B(1,\lambda)$ into its socle.  Therefore, ${\text{Hom}}_{\Lambda}(P_3,B(1,\lambda)) \circ {\text{Hom}}_{\Lambda}(B(1,\lambda),P_3) = 0$ if and only if $\lambda^2 = -1$.  Since ${\text{End}}_{\Lambda}(B(1,\lambda))$ is two-dimensional, and the identity map clearly does not factor through a projective, the result follows for $n=1$.  

Now consider $n\geq 2$ and $1 \leq i \leq e-1$. Since $n \geq 2$, it follows $B(1,\lambda)$ is isomorphic to both a submodule and a quotient of $B(n,\lambda)$.  Thus, let $t: B(n,\lambda) \to B(n,\lambda)$ denote projection onto $B(1,\lambda)$ followed by inclusion.  Since $im(t) \nsubseteq \textrm{rad}(B(n,\lambda))$, $t$ does not factor through any projective by Lemma \ref{heightlemma}, as ${\textrm{rad}}^{e+1} \cdot B(n,\lambda) = 0$.  Thus, $\textrm{dim}_k({\underline{\text{End}}}_{\Lambda} (B(n,\lambda))) \geq 2$ and the result follows.
 
\end{pf}

\begin{lemma}\label{bandextlemma}
Let $\lambda \in k^*$.  Then, $\displaystyle \Omega(B(1,\lambda))=B(1,-{\lambda}^{-1})$. 
\end{lemma}

\begin{pf}
Let $\lambda \in k^*$. Define bases $\left\{b_{i,j}\right\}$ of $P_e$ and $\left\{c_i\right\}$ of $B(1,\lambda)$, and the natural projection $s_2:P_e \to B(1,\lambda)$ as in Lemma \ref{bandstablelemma}. Since $P_e$ is the projective cover of $B(1,\lambda)$, $ker(s_2) \cong \Omega(B(1,\lambda))$. A basis for $ker(s_2)$ is then given by $\left\{d_i\right\}$ where $d_e=(b_{e,2}-\lambda^{-1}b_{e,3}), d_{e+1}=b_{e,4}$ and $d_i = b_{i,2}$ for $1 \leq i \leq e-1.$ Note that $\delta(d_e) = \delta(b_{e,2}-\lambda^{-1}b_{e,3}) = -\lambda^{-1}b_{e,4} = -\lambda^{-1}\alpha_{e-1} \ldots \alpha_1 \alpha_e(b_{e,2}-\lambda^{-1}b_{e,3})= -\lambda^{-1}\alpha_{e-1} \ldots \alpha_1 \alpha_e(d_e).$ It can be further verified that the $kQ_e/I$-module structure of $ker(s_2)$ is that of $B(1,-\lambda^{-1})$, and therefore $\Omega(B(1,\lambda)) \cong ker(s_2) \cong B(1,-\lambda^{-1})$.

\end{pf}

\section{Proof of Main Results}
\label{MainResults}

\label{sec main results}

Let $k$ be an algebraically closed field of arbitrary characteristic, and $\Lambda$ be a 1-domestic symmetric special biserial algebra.  By the argument in Section \ref{brauer} we reduce to the case of $\Lambda=kQ_e/I$, where $Q_e$ is given if Figure \ref{fig:OurQuiver}.  We now finish with some propositions.

\begin{prop} \label{exceptional}
Let $\mathfrak{C}$ be an exceptional tube of the stable Auslander-Reiten quiver of $\Lambda$.

\begin{enumerate}[(i)]

\item If $e \geq 2$, then there are precisely $e$ $\Omega$-orbits of modules in $\mathfrak{C}$  with stable endomorphism ring isomorphic to $k$. Moreover, $R(\Lambda,V) \cong k$ for the modules belonging to $e-1$ of these $\Omega$-orbits, and $R(\Lambda,V) \cong k[[t]]$ for the modules belonging to the remaining $\Omega$-orbit.

\item If $e = 1$, there is precisely 1 $\Omega$-orbit of modules in $\mathfrak{C}$  with stable endomorphism ring isomorphic to $k$, and $R(\Lambda,V) \cong k[[t]]$ for the modules belonging to this $\Omega$-orbit. 
\end{enumerate}
\end{prop}

\begin{pf}

First, consider $e \geq 2$. We recall by Lemma \ref{stringmods} that the modules with stable endomorphism ring isomorphic to $k$ are the string modules $M(1_{e-1}),$ denoted by $S_{e-1}$, $M(\alpha_{e-2} \cdots \alpha_i)$ for $1 \leq i \leq e-2$, and $M(\alpha_{e-2} \cdots \alpha_1 \alpha_e \delta^{-1})$, along with their $\Omega$-translates. 

Consider the simple module $S_{e-1}$. Recall $\text{Ext}_\Lambda^1(M,M) \cong \text{\underline{Hom}}_\Lambda(\Omega(M),M)$. Since $\Omega(S_{e-1}) = M(\alpha_{e-2} \alpha_{e-3} \ldots \alpha_1 \alpha_{e} \ldots \alpha_1 \alpha_e)$ it follows that $\Omega(S_{e-1})/\textrm{rad}(\Omega(S_{e-1})) \cong S_e \not\cong S_{e-1}$, and therefore $\text{Hom}_\Lambda(\Omega(S_{e-1}),S_{e-1})=0$. Therefore, $\text{Ext}_\Lambda^1(S_{e-1},S_{e-1}) \cong  \text{\underline{Hom}}_\Lambda(\Omega(S_{e-1}),S_{e-1}) = 0$, which implies $R(\Lambda, S_{e-1}) \cong k$.

Let $V_i = M(\alpha_{e-2} \alpha_{e-3} \ldots \alpha_{i})$ for any $i \in \{1, \ldots, e-2\}$. Then $\Omega(V_i) = M(\alpha_{i-1} \ldots \alpha_1 \alpha_e \ldots \alpha_1 \alpha_e)$. Since $\Omega(V_i)/\textrm{rad}(\Omega(V_i)) \cong S_e$, which does not appear as a composition factor of $V_i$, it follows that $\text{Hom}_\Lambda(\Omega(V_i),V_i)=0$. Therefore $\text{Ext}_\Lambda^1(V_i,V_i) \cong \text{\underline{Hom}}_\Lambda(\Omega(V_i),V_i) = 0$, which implies $R(\Lambda, V_i) \cong k$ for all $1 \leq i \leq e-2$.

Let $W = M(\alpha_{e-2}\alpha_{e-3} \ldots \alpha_1 \alpha_e \delta^{-1}).$ Then $\Omega(W) = M(\alpha_{e-1}\alpha_{e-2} \ldots \alpha_1 \alpha_e)$. In this case, it follows $\text{Ext}_\Lambda^1(W,W) \cong \text{\underline{Hom}}_\Lambda(\Omega(W),W) \cong k$, as the map which sends $\Omega(W)/\textrm{rad}(\Omega(W)) \cong S_e$ to $\textrm{soc}(W) \cong S_e$ does not factor through a projective and generates $\text{\underline{Hom}}_\Lambda(\Omega(W),W)$. Therefore, $R(\Lambda, W) \cong k[[t]]/J$ for some ideal $J$. To show $R(\Lambda, W) \cong k[[t]]$, it is enough to prove $W$ has a lift $M$ over $k[[t]]$ such that $M/t^2M$ is a non-trivial lift over $k[[t]]/t^2$.  To do so, define $M$ to be the free $k[[t]]$-module of rank $e+1$ where the arrows of $Q_e$ act via the following matrices: $\alpha_{i} = \begin{cases}  E_{i+1,i} & 1\leq i \leq e-2 \\ tE_{e,e-1} & i=e-1 \\  E_{1,e}  & i=e,  \end{cases}$ and  $\delta = E_{e+1, e}$, where $E_{j,i}$ denotes the matrix sending the $i$th basis element to the $j$th, and all other basis elements to $0$. $M$ is a $k[[t]] \otimes \Lambda$-module, which is free as a $k[[t]]$-module, $M/tM \cong W$, and $M/t^2M$ is a non-trivial lift of $W$ over $k[[t]]/t^2$. Therefore $R(\Lambda, W) \cong k[[t]]$.

Now consider $e=1.$ Again by Lemma \ref{stringmods}, the modules with stable endomorphism ring isomorphic to $k$ are the string modules $M(\alpha)$ and $M(\delta) = \Omega((\alpha))$. We then have
 $\text{Ext}_\Lambda^1(M(\alpha),M(\alpha)) \cong \text{\underline{Hom}}_\Lambda(\Omega(M(\alpha)),M(\alpha)) \cong k$, since the map sending $M(\delta)/\textrm{rad}(M(\delta)) \cong S_1$ to $\textrm{soc}(M(\alpha)) \cong S_1$ generates $\underline{\text{Hom}}_\Lambda(M(\delta)),M(\alpha))$. 

Therefore, $R(\Lambda, M(\alpha)) \cong k[[t]]/J$ for some ideal $J$. To show $R(\Lambda, (\alpha)) \cong k[[t]]$, define $L$ to be the free $k[[t]]$-module of rank $2$ where the arrows of $Q_e$ act via the following matrices: $\alpha = E_{2,1}$ and  $\delta = tE_{2, 1}$. $L$ is a $k[[t]] \otimes \Lambda$-module, which is free as a $k[[t]]$-module, $L/tL \cong M(\alpha)$, and $L/t^2L$ is a non-trivial lift of $M(\alpha)$ over $k[[t]]/t^2$. Therefore $R(\Lambda, M(\alpha)) \cong k[[t]]$.

\end{pf}

\begin{prop} \label{nonperiodic}
Let $\mathfrak{C}$ be the non-periodic component of the stable Auslander-Reiten quiver of $\Lambda$. Then there are precisely $e$ $\Omega$-orbits of modules in $\mathfrak{C}$, and all modules in these $\Omega$-orbits have stable endomorphism ring isomorphic to $k$.

\begin{enumerate}[(i)]

\item If $e \geq 2$, then $R(\Lambda,V) \cong k$ for the module belonging to $e-2$ of these $\Omega$-orbits, and $R(\Lambda,V) \cong k[[t]]/(t^2)$ for the modules belonging to the remaining 2 $\Omega$-orbits.

\item If $e = 1$, then $R(\Lambda,V) \cong k[[t_1,t_2]]/(t_1^2-t_2^2,t_1t_2)$ for every module belonging to $\mathfrak{C}$.
\end{enumerate}
\end{prop}

\begin{pf}

In Lemma \ref{nonperiodicmods} we established the modules with stable endomorphism ring isomorphic to $k$.

Assume $e > 1$. Consider the simple module $S_{e}$.  Then $\text{Ext}_\Lambda^1(S_{e},S_{e}) \cong  \text{\underline{Hom}}_\Lambda(\Omega(S_{e}),S_{e}) \cong k$, since $\Omega(S_{e})/\textrm{rad}(\Omega(S_{e})) \cong S_e $, and the projection of $\Omega(S_{e})$ onto $S_e$ generates $\underline{\text{Hom}}_\Lambda(\Omega(S_{e}),S_{e})$.

Therefore, $R(\Lambda, S_e) \cong k[[t]]/J$ for some ideal $J$. To show $R(\Lambda, S_e) \cong k[[t]]/(t^2)$, define $M$ to be the free $k[[t]]/(t^2)$-module of rank $1$ where the arrows of $Q_e$ act via the following matrices: $\alpha_i = 0$ for $1 \leq i \leq e$ and $\delta = [t]$. $M$ is a $k[[t]]/(t^2) \otimes \Lambda$-module and $M/tM \cong S_e$.

Suppose now that $R(\Lambda, S_e) \not\cong k[[t]]/(t^2)$. Then there must exist a non-trivial lift of $S_e$ over a ring of the form $k[[t]]/J$ for some ideal $J \subset (t^2)$. In particular, there must be a lift $M'$ to  $k[[t]]/(t^3)$ where $M'/t^2M' \cong M$. For this to be true, the arrows of $Q_e$ must act on $M'$ via matrices which reduce to the matrices determining the lift $M$ modulo the ideal $(t^2)$. Additionally, every composition factor of $M'$ as a $\Lambda$-module is isomorphic to $S_e$. Therefore, the arrows of $Q_e$ act on $M'$ via the following matrices: $\alpha_i = 0$ for $1 \leq i \leq e$ and $\delta = t+d$ for some $d \in (t^2)$. 

Since $M'$ is a $k[[t]]/(t^3) \otimes \Lambda$-module, the above matrices for $M'$ must satisfy the relation \\$(\alpha_{e-1} \alpha_{e-2} \ldots \alpha_1 \alpha_e)^2 \equiv \delta^2$ mod $(t^3)$.
But since $\alpha_i =0$ for all $1 \leq i \leq e$, it follows that $0 \equiv \delta^2 = (t+d)^2 = t^2+2dt+d^2 \equiv t^2$ mod $(t^3)$, since $d \in (t^2)$. This gives $t^2 \in (t^3)$, a contradiction, and therefore,  $R(\Lambda, S_e) \cong k[[t]]/(t^2)$.

Let $V=M(\alpha_{e-1} \alpha_{e-2} \ldots \alpha_1)$. Then  $\text{Ext}_\Lambda^1(V,V) \cong  \text{\underline{Hom}}_\Lambda(\Omega(V),V) \cong k$, since the natural projection of $\Omega(V) = M(\alpha_e \alpha_{e-1} \ldots \alpha_1)$ onto $V$ generates $\underline{\text{Hom}}_\Lambda(\Omega(V),V)$.

Therefore, $R(\Lambda, V) \cong k[[t]]/J$ for some ideal $J$. To show $R(\Lambda, V) \cong k[[t]]/(t^2)$, define $N$ to be the free $k[[t]]/(t^2)$-module of rank $e$ where the arrows of $Q_e$ act via the following matrices: $\alpha_i = E_{i+1,i}$ for $1 \leq i \leq e-1$, $\alpha_e = tE_{1,e}$ and $\delta = 0$. $N$ is a $k[[t]]/(t^2) \otimes \Lambda$-module and $N/tN \cong V$. This is the Universal deformation of $V$, which is proven similarly to the previous case for $S_{e-1}$. Therefore $R(\Lambda, V) \cong k[[t]]/(t^2)$.

Let $W_i = M(\alpha_{e-1} \alpha_{e-2} \ldots \alpha_i)$ for any $1 \leq i \leq e-1$. Then $\text{Ext}_\Lambda^1(W_i,W_i) \cong  \text{\underline{Hom}}_\Lambda(\Omega(W_i),W_i) = 0$, since $\Omega(W_i)/\textrm{rad}(\Omega(W_i)) \cong S_1,$ which is not a composition factor of $W_i$. Therefore $R(\Lambda, V) \cong k$ for all $1 \leq i \leq e-1$.

Now assume $e=1$ and consider the simple module $S_1$. It follows that $\Omega(S_1) \cong M(\alpha^{-1} \delta)$, and $M(\alpha^{-1} \delta)/\textrm{rad}(M(\alpha^{-1} \delta)) \cong S_1 \oplus S_1$. Thus, there are two linearly independent maps from $M(\alpha^{-1} \delta)$ to $S_1$ given by the natural projections of the simple summands of $M(\alpha^{-1} \delta)/\textrm{rad}(M(\alpha^{-1} \delta))$ onto $S_1$, neither of which factor through a projective. Therefore, $dim(\text{Ext}_\Lambda^1(S_{1},S_{1})) = 2$, giving that $R(\Lambda, S_1) \cong k[[t_1,t_2]]/J$ for some ideal $J$.

To prove $R(\Lambda, S_1) \cong k[[t_1,t_2]]/(t_1^2-t_2^2, t_1t_2)$, define $L$ to be a free $k[[t_1, t_2]]/(t_1^2-t_2^2, t_1t_2)$-module of rank $1$ where the arrows of $Q_e$ act via the following matrices: $\alpha = [t_1], \delta=[t_2]$. $L$ is a $k[[t_1, t_2]]/(t_1^2-t_2^2, t_1t_2) \otimes \Lambda$-module and $L/(t_1,t_2)L \cong S_1$. Furthermore, $L/t_1L$ and $L/t_2L$ are two non-trivial lifts of $S_1$ over the dual numbers. Additionally, they are non-isomorphic, as $L/t_1L \cong M(\delta)$ and $L/t_2L \cong M(\alpha)$ as $k\Lambda$-modules.

Suppose now that $R(\Lambda, S_1) \not\cong k[[t_1,t_2]]/(t_1^2-t_2^2, t_1t_2)$. Then there must exist a non-trivial lift of $S_1$ over a ring of the form $k[[t_1,t_2]]/J$ for some ideal $J \subset (t_1^2-t_2^2, t_1t_2)$. In particular, there must be a lift $L'$ to  $k[[t_1,t_2]]/J'$ where $(t_1,t_2)(t_1^2-t_2^2, t_1t_2) \subseteq J'$ and $L'/(t_1^2-t_2^2, t_1t_2)L' \cong L$. This implies that the arrows of $Q_e$ must act on $L'$ via matrices which reduce to the matrices determining the lift $L$ modulo the ideal $(t_1^2-t_2^2, t_1t_2)$. Therefore, the arrows of $Q_e$ act on $L'$ via the following matrices: $\alpha = [t_1+a]$ and $\delta = [t_2+d]$ for some $a,d \in (t_1^2-t_2^2, t_1t_2)$. 

Since $L'$ is a $k[[t_1,t_2]]/J' \otimes \Lambda$-module, the above matrices must satisfy the relations $\alpha \delta \equiv 0$ mod $J'$ and $\alpha^2 - \delta^2 \equiv 0$ mod $J'$. The first relations gives $0 \equiv (t_1+a)(t_2+d) = t_1t_2+t_1d+t_2a+ad \equiv t_1t_2$ mod $J'$, since $t_1d, t_2a, ad \in J'$. The second relation gives $0 \equiv (t_1+a)^2 - (t_2+d)^2 \equiv t_1^2-t_2^2$ mod $J'$ since $t_1a, a^2, t_2d, d^2 \in J'$. Therefore $(t_1^2-t_2^2, t_1t_2) \subset J'$, which gives a contradiction. Therefore, $R(\Lambda, S_1) \cong k[[t_1,t_2]]/(t_1^2-t_2^2, t_1t_2)$ .

\end{pf}

\begin{prop} \label{homogeneous}
Let $\mathfrak{C}$ be a homogeneous tube of rank one of the stable Auslander-Reiten quiver of $\Lambda$. Then there is at most 1 $\Omega$-orbit of modules which have stable endomorphism ring isomorphic to $k$, and $R(\Lambda,V) \cong k[[t]]$ for all modules belonging to this $\Omega$-orbit.

\end{prop}

\begin{pf} 

The modules in the homogeneous tubes of rank one are all of the form $B(n,\lambda)$ for $n \in \mathbb{Z}^+$ and $\lambda \in k^*$. By Lemma \ref{bandstablelemma}, ${\underline{\text{End}}}_{\Lambda} (B(n,\lambda)) \cong k$ if and only if $n=1$ and $\lambda^2\neq -1$. Therefore, we now consider $(B(1,\lambda))$ with $\lambda^2\neq -1$. 

By Lemma \ref{bandextlemma}, $\displaystyle \Omega(B(1,\lambda))=B(1,-{\lambda}^{-1})$. It can be easily shown that $\text{Ext}_\Lambda^1(B(1,\lambda),B(1,\lambda)) \cong \underline{\text{Hom}}_\Lambda(B(1,-{\lambda}^{-1}),B(1,\lambda)) \cong k$, as $\text{Hom}_\Lambda(B(1,-{\lambda}^{-1}),B(1,\lambda)) \cong k$ and the map spanning this space will not factor through a projective using a similar argument to that used in Lemma \ref{bandstablelemma} for the case when $n=1$.

Therefore, $R(\Lambda, B(1,\lambda)) \cong k[[t]]/J$ for some ideal $J$. To show $R(\Lambda, B(1,\lambda)) \cong k[[t]]$, define a basis $\left\{c_i\right\}$ of $B(1,\lambda)$ as in Lemma \ref{bandstablelemma}. Define $L$ to be the free $k[[t]]$-module of rank $2$ where the arrows of $kQ_e/I$ act via the following matrices: $\alpha_i = E_{i+1,i}$ for all $1 < i< e-1$, $\alpha_{e-1} = E_{e+1,e-1}$, $\alpha_e = E_{1,e}$,  and  $\delta = E_{(t+\lambda)(e+1), e}$ ,where $E_{j,i}$ denotes the matrix sending $c_i$ to the $c_j$, and all other basis elements to $0$. $L$ is a $k[[t]] \otimes \Lambda$-module, which is free as a $k[[t]]$-module, $L/tL \cong B(1,\lambda)$, and $L/t^2L$ is a non-trivial lift of $B(1,\lambda)$ over $k[[t]]/t^2$. Therefore $R(\Lambda, M(\alpha)) \cong k[[t]]$.

\end{pf}

Thus, Theorem \ref{mainthm} is proven.
\appendix

\section{String and Band Modules}
\label{appendix}

Let $k$ be an algebraically closed field, and let $\Lambda = kQ/I$ be a generalized Brauer tree algebra. Then $\bar{\Lambda}=\Lambda/\text{soc}(\Lambda)=kQ/J$ is a string algebra and all non-projective indecomposable $\Lambda$-modules are inflated from string and band modules for this associated string algebra. 

Given an arrow $\beta$ of $Q$, let $\beta^{-1}$ denote a formal inverse of $\beta$, and define $s(\beta^{-1}) = e(\beta)$, $e(\beta^{-1}) = s(\beta)$, and $(\beta^{-1})^{-1} = \beta$. 
 A word of length $n \geq 1$ is defined to be a sequence $w=w_1 w_2 \cdots w_n$ where each $w_i$ is either an arrow or a formal inverse, and where $s(w_i) = e(w_{i+1})$ for $1 \leq i \leq n-1$. Define $s(w) = s(w_n)$, $e(w) = e(w_1)$, and $w^{-1} = w_n ^{-1} \cdots w_2 ^{-1} w_1 ^{-1}$. We also associate an emptry word $1_u$ of length $0$ to each vertex where $e(1_u)=s(1_u)=u$ and $(1_u)^{-1}=1_u.$
 
\begin{definition}
Define an equivalence relation $\sim$ on the set of all words by $w \sim w^{\prime}$ if $w=w^{\prime}$ or $w^{-1} = w^{\prime}$. Define a second equivalence relation $\sim_b$ on the set of all words $w$ of length at least $1$ which satisfy $s(w) = e(w)$ by $w \sim_b w^{\prime}$ if either $w$ or $w^{-1}$ is obtained from $w^{\prime}$ by a cyclic rotation.

Let $\mathcal{S}$ be a complete set of representative words $w = w_1 w_2 \cdots w_n$ under the relation $\sim$ such that either $n = 0$, i.e. $w$ is an empty word, or $w_i \neq w_{i+1} ^{-1}$ for $1 \leq i \leq n-1$, and no subpath of $w$ or $w^{-1}$ belongs to $J$. The elements of $\mathcal{S}$ are called strings.

Let $\mathcal{B}$ be a complete set of representative words $w = w_1 w_2 \cdots w_n$ with $n \geq 1$ and $s(w) = e(w)$ under the relation $\sim_b$ such that $w_i \neq w_{i+1}^{-1}$ for $1 \leq i \leq n-1$, $w_n \neq w_1 ^{-1}$, $w$ is not a power of a smaller word, and no subword of $w^m$ belongs to $J$ for all $m \geq 1$. The elements of $\mathcal{B}$ are called bands.
\end{definition}

If $C=1_u$ for some vertex $u \in Q$, then the associated string module $M(1_u)=S_{u}$ is the simple $\Lambda$-module corresponding to $u$. Otherwise, Let $C = w_1 w_2 \cdots w_n$ or $w= e_i$ be a string of length $n > 0$ and define $v(i)=e(w_{i+1})$ and $v(n) = s(w_n)$. The string module $M(C)$ is defined with the ordered basis $\{b_0, b_1, \ldots , b_n\}$ where the $\Lambda$-action on $Q$ is given on this basis with each vertex $u$, respectively arrow $\beta$, acting as an $(n+1) \times (n+1)$ $X_u$, respectively $X_\beta$, by the following: $X_u$ sends $b_i$ to itself if $v(i)=u$ and $0$ overwise, and $X_\beta$ sends $b_i$ to $b_{i-1}$ if $w_{i+1}=\beta$, to $b_{i+1}$ if $w_i=\beta^{-1}$ and to $0$ otherwise.

In this paper, we consider the algebra $\Lambda = kQ_e/I$ associated to the generalized Brauer tree as described in Section \ref{brauer}. By lemma \ref{stringmods} and lemma \ref{nonperiodicmods}, the only string modules for $\Lambda$ with stable endomorphism ring isomorphic to $k$ are the below string modules and their $\Omega$-orbits:
\begin{itemize}
    \item for all $e\geq 1$, the simple module $S_e$  and the uniserial modules $M(\alpha_{e-1} \alpha_{e-2} \cdots \alpha_i)$ for $1 \leq i \leq e-1,$ 
    \item additionally, when $e>1$, the string modules $S_{e-1}$, $M(\alpha_{e-2}\alpha_{e-3} \cdots \alpha_i)$ for $1\leq i\leq e-2,$ and $M(\alpha_{e-2} \alpha_{e-3} \cdots \alpha_1 \alpha_e \delta^{-1})$,
    \item and for $e = 1$, the string module $M(\alpha)$.
\end{itemize}

In \cite{krause1991maps}, the author describes the homomorphisms with string and band modules. The following summarizes those definitions as used in our paper.

\begin{definition}
Let $M(S)$ and $M(T)$ be string modules for $\bar{\Lambda}$ with $k$ bases 
$\{b_0, \ldots, b_n\}$ and $\{c_0, \ldots, c_m\}$ respectively with the following properties:

\begin{enumerate}[(i)]
\item $S \sim BCD$, where $B$ is a substring which is either of length $0$ or $B = B^{\prime} \tau$ for an arrow $\tau$, and $D$ is a substring which is either of length $0$ or $D = \sigma ^{-1} D^{\prime}$ for an arrow $\sigma$. In other words, $S \sim B^{\prime} \stackrel{\tau}{\longleftarrow} C \stackrel{\sigma}{\longrightarrow} D^{\prime} $.
\item $T \sim ECF$, where $E$ is a substring which is either of length $0$ or $E = E^{\prime} \nu^{-1}$ for an arrow $\nu$, and $F$ is a substring which is either of length $0$ or $F = \mu F^{\prime}$ for an arrow $\mu$. In other words, $T \sim E^{\prime} \stackrel{\nu}{\longrightarrow} C \stackrel{\mu}{\longleftarrow} F^{\prime} .$
\end{enumerate}
Then there exists a canonical $\overline{\Lambda}$-module homomorphism 
$\phi_C : M(S) \twoheadrightarrow M(C) \hookrightarrow M(T)$ which is defined by sending each basis element in $\{b_1, \ldots , b_n\}$ of $S$ either to $0$ or a basis element in $\{ c_0, \ldots, c_m \}$ of $T$ corresponding to the same vertex $u$ of the quiver $kQ_e$.
\end{definition}

There can be several strings $B, D, E$ and $F$ giving valid and linearly independent canonical homomorphisms for a given string $C$, and by \cite{krause1991maps}, every $\bar{\Lambda}-$module homomorphism $\phi: M(S) \rightarrow M(T)$ can be written as a unique linear combination of these canonical homomorphisms for multiple corresponding strings $C$.

For $\Lambda = kQ_e/I$, the only band modules correspond to the word $w=\alpha_{e-1} \alpha_{e-2} \cdots \alpha_1 \alpha_e \delta^{-1}$. In section \ref{bandmodules} we give an explicit description of these band modules as representations of $\Lambda$ (see Figure \ref{fig:BandAsRep}) and we also describe all morphisms as morphisms between representations of $\Lambda$.

\section*{Acknowledgments}
The authors wish to acknowledge Frauke Bleher both for introducing us to this field of study, and for all of her guidance.

\bibliography{brauer_graph_universal.bib}
\bibliographystyle{elsarticle-num}

\end{document}